\documentclass[a4paper,12pt, twoside, reqno]{amsart}
\usepackage{amssymb}
\usepackage{amsmath}
\usepackage{tikz}
\usetikzlibrary{arrows}
\tikzstyle{block}=[draw opacity=0.7,line width=1.4cm]
\usetikzlibrary{calc}
\usepackage{tkz-base,tkz-euclide}
\usepackage[active]{srcltx}
\usepackage{enumerate}
\usepackage{graphicx, float}

\newtheorem{theorem}{Theorem}

\newtheorem{definition}{Definition}
\newtheorem{example}{Example}
\newtheorem{remark}{Remark}[section]

\title{Recent developments in the fixed point theory of enriched contractive mappings. A survey}
\author{Vasile BERINDE$^{1,2}$}
\author{M\u ad\u alina P\u ACURAR$^{3}$}
\begin{document}
\maketitle \pagestyle{myheadings} \markboth{Vasile Berinde and M\u ad\u alina P\u ACURAR} {Enriched Kannan type...}
\begin{abstract}
The aim of this note is threefold: first, to present a few relevant facts about the way in which the technique of enriching contractive mappings was introduced; secondly, to expose the main contributions in the area of enriched mappings established by the authors and their collaborators by using this technique; and third, to survey some related developments in the very recent literature which were authored by other researchers.

\end{abstract}

\section{Introduction}

The concept of {\em enriched} nonexpansive mapping has been introduced in Berinde \cite{Ber19} in the case of a real Hilbert space and then was extended to the more general case of a Banach space in Berinde \cite{Ber20}, see also Berinde \cite{Ber18} where the technique of {\em enriching contractive type mappings} has been applied to strictly pseudocontractive operators. 

Soon after that, the authors and their collaborators applied successfully the same technique for some other classes of contractive type mappings in Hilbert spaces, Banach spaces or convex metric spaces, see Berinde \cite{Ber22}, \cite{Ber22b}, \cite{Ber23}, Berinde and P\u acurar \cite{BPac20}, \cite{BPac21}, \cite{BPac21a}, \cite{BPac21b}, \cite{BPac21c}, \cite{BPac21d}, \cite{BPac22}, Berinde et al. \cite{Ber22a}, Abbas et al. \cite{Abb21}, \cite{Abb21a}, Salisu et al. \cite{Sali23a},... 

Many other authors were attracted to work in the same area and therefore some interesting developments on enriched mappings were obtained, most of them included in the list of References.

The impressive interest for the use of the technique of enriching contractive type mappings suggested to us to undertake the task of offering a comprehensive exposure  to date on the subject. 
So, our main aim in this paper is threefold: 
\begin{enumerate}
\item to present a few relevant facts about the way in which the technique of enriching contractive mappings was (re-)discovered;
\item to expose the main contributions in the area of enriched mappings established by the authors and their collaborators  by using this technique;
\item to survey some related developments in this context which were authored by other researchers.
\end{enumerate}
  
 %Having in view the considerable interest on the study of enriched contractive type mappings  in the last five years or so, the aim of this note is to expose some of the main developments that can be found in literature.
 
 The paper is organized as follows: in Section \ref{s2} we give a brief account on how the technique of enriching contractive type mappings has been (re)-discovered 
and   present the facts about the origins of the technique of enriching contractive type mappings. 

Section \ref{s3}  is devoted to the exposition of some classes of enriched mappings, Section \ref{s4} exposes the classes of enriched nonexpansive mappings in Hilbert and Banach spaces, Section \ref{s5}  gives an account on the concepts of unsaturated and saturated classes of contractive mappings, Section \ref{s6}  surveys the brand new results on enriched contractions in quasi-Banach spaces, while Section \ref{s7}  deals with other developments in the area of enriched contractive type mappings by other authors.
 
 \section{The technique of enriching contractive type mappings}\label{s2}
 
 We start by presenting how the technique of enriching contractive type mappings was discovered (in fact, re-discovered, see explanations following later in this section). 
 
 For a rather long period of time, in connection with the study of various fixed point iterative schemes, partially surveyed in the monograph Berinde \cite{Ber07}, we were thinking about finding a way to compare the very many classes of nonexpansive type mappings existing in literature, as such a complete comparison did not exist.
 
 Basically, at the beginning, we were trying to compare, by appropriate examples, the following four important classes of nonexpansive type mappings:
 \begin{itemize}
\item nonexpansive mappings
\item quasi  nonexpansive mappings
\item strictly pseudocontractive mappings
\item demicontractive mappings,
\end{itemize}
 as at that time we were interested to deepen our knowledge on the great generality and value of demicontractive mappings. This  task  has been completed after a while and was very recently published in the paper Berinde \cite{Ber24}.
 
 On the way of performing such a comparison, we discovered by chance a method of deriving new constructive fixed point theorems, which we have called, based on the arguments presented above, as the technique of {\em enriching contractive type mappings}. 
 
 The starting point came from some known facts in the metrical fixed point theory. To present them, let $(X, \|\cdot\|)$ be a real normed space, $C\subset X$ a closed and convex set and $T:C\rightarrow C$ a self mapping. Denote by
$$
Fix\,(T)=\{x\in C: Tx=x\},
$$
the set of fixed points of  $T$. For $\lambda\in (0,1)$, let us also denote
$$
T_\lambda:=(1-\lambda) I+\lambda T.
$$
$T_\lambda$ is usually named as the {\em averaged} perturbation of $T$.  It is easy to see that
\begin{equation}\label{b1}
Fix\,(T)=Fix(T_\lambda)
\end{equation}
for all $\lambda\in (0,1)$.

A mapping $T$ is said to be
 {\it nonexpansive} if 
\begin{equation}\label{ne}
\|Tx-Ty\|\leq \|x-y\|,\,\textnormal{ for all } x,y\in C.
\end{equation}

We also recall that a mapping  $T:C\rightarrow C$ is called \textit{asymptotically regular} (on $C$)
if, for any $x\in C$, 
$$
\|T^{n+1}x-T^n x\|\rightarrow 0  \textnormal{ as } n\rightarrow \infty.
$$ 
It is known that if $T$ is nonexpansive, then  in general $T$ is not asymptotically regular but its averaged perturbation, $T_{\lambda}$, is asymptotically regular, a result that apparently was first established by Krasnoselskii \cite{Kra55}, in uniformly convex Banach spaces, and then used and developed by Browder and Petryshyn \cite{BroP66}, \cite{BroP67} in Hilbert spaces. 

This property is extremely important as it enables us to compute the fixed points of a nonexpansive mapping $T$ by means of its  averaged perturbation $T_\lambda$, which,  from the point of view of the convergence of the iterations, has  {\bf richer} properties than $T$. %(but they share the same set of fixed points) 

After contemplating for a long time the above enriching property of nonexpansive mappings with respect to asymptotical regularity, by using instead of $T$ its averaged perturbation $T_\lambda$, we were naturally conducted to formulate  the following
\medskip

{\bf Open problem}:   If one uses $T_\lambda =(1-\lambda) I+\lambda T$ instead of $T$ in a certain contraction conditions from metrical fixed point theory, do we obtain a richer classes of mappings ?
\medskip

Fortunately, the answer was in the affirmative: indeed, by using $T_\lambda =(1-\lambda) I+\lambda T$ instead of $T$ in some contraction conditions from metrical fixed point theory, we obtained richer classes of mappings, see the results surveyed in Sections \ref{s3}-\ref{s6}.

We first searched an answer for the above question in the case of nonexpansive mappings, by considering inequality \eqref{ne} with $T_\lambda$ instead of $T$, that is, by introducing 
\begin{equation}\label{b2}
\|T_\lambda x-T_\lambda y\|\leq \|x-y\|,\,\textnormal{ for all } x,y\in C,
\end{equation}
and thus we identified  the class of {\em enriched nonexpansive mappings} in Hilbert spaces (in Berinde \cite{Ber19}) and called such  a mapping $T$ as being  {\em enriched nonexpansive}. Its definition will be given later in Section \ref{s3}.
 
Of course, at the very first steps, we were convinced that we were the first ones to discover this nice technique but, a few years later, after a careful documentation and analysis, we realized that the same technique has been applied independently and tacitly by other mathematicians long time ago, e.g., Browder and Petryshyn \cite{BroP67} and Hicks and Kubicek \cite{Hicks}, see Berinde and P\u acurar \cite{BPac21d}. 
 
 Without explicitly indicating thei method of derivation, Browder and Petryshyn \cite{BroP67} introduced and studied the class of {\em strictly pseudocontractive mappings}, while Hicks and Kubicek \cite{Hicks}  introduced and studied the class of {\em demicontractive mappings}. These classes represent two important concepts   in the iterative approximation of fixed points, see for example Berinde \cite{Ber07}. 
 Recall that a mapping $T$ is said to be

1) {\it quasi-nonexpansive} if $Fix\,(T)\neq \emptyset$ and 
\begin{equation}\label{qne}
\|Tx-y\|\leq \|x-y\|,\, \textnormal{ for all } x\in C \textnormal{ and } y\in Fix\,(T).
\end{equation}

2) {\it $k$-strictly pseudocontractive} of the Browder-Petryshyn type (\cite{BroP66}) if there exists $k<1$ such that
\begin{equation}\label{strict}
\|Tx-Ty\|^2\leq \|x-y\|^2+k\|x-y-Tx+Ty\|^2, \forall x,y\in C.
\end{equation}

3) {\it $k$-demicontractive} (\cite{Hicks}) or {\it quasi $k$-strictly pseudocontractive}  (see Berinde et al. \cite{BPR23}) if $Fix\,(T)\neq \emptyset$ and there exists  a positive number $k<1$ such that 
\begin{equation}\label{demi}
\|Tx-y\|^2\leq \|x-y\|^2+k\|x-Tx\|^2,
\end{equation}
for all $x\in C$ and $y\in Fix\,(T)$. %We also say that $T$ is $k$-demicontractive.
%\end{definition}
\bigskip

If we denote by $\mathcal{NE}$, $\mathcal{QNE}$, $\mathcal{SPC}$ and $\mathcal{DC}$ the classes of nonexpansive, quasi-nonexpansive, strictly pseudocontractive and demicontractive mappings, respectively, then the relationships between these classes are completely represented in the following diagram
\medskip

\begin{tikzpicture}

\draw[red, thick] node [anchor=south east]{NE} node [anchor=south west]{$\cdot \,T_2$} (-2,0) node [anchor=south west]{$\cdot \,T_1$} rectangle (1,2);
   
   \draw[blue, thick]  (-2.3,-1)  node [anchor=south west]{$\cdot \,T_3$}  rectangle (2,3) node [anchor=north east]{SPC};

      \draw[black, thick, dashed]  (-0.99,-2) node [anchor=south west]{$\cdot \,T_4$} rectangle   (4,0.99) node [anchor=north east]{QNE};
      
     \draw[green, thick]  (-1,-3) node [anchor=south west]{$\cdot \,T_5$} rectangle  (6,1)  node [anchor=north east]{DC} ;

  %\caption{Diagram of the relationship between the classes $\mathcal{NE}$, $\mathcal{QNE}$, $\mathcal{SPC}$ and $\mathcal{DC}$}  
       
 % \draw[red,very thick,dashed] (-2,-1) -- (1,-1)  -- (1,1) -- (0,1) -- (0,2) -- (-2,2) node [anchor=north west]{ENE}  -- (-2,-1) node [anchor=south west]{$\cdot \,T_3$}  -- cycle;
 % \pause

%\draw[blue,very thick,dashed] (0,0) -- (0,1) -- (-1,1) -- (-1,2) -- (3,2) -- (3,0) node [anchor=south east]{DC} --  cycle;
%\pause

%  \draw[green,very thick,dashed] (-1,-1) -- (-1,1) -- (0,1) -- (0,2) -- (2,2) -- (2,-1) node [anchor=south east]{CRR}  -- (-1,-1) -- cycle;

\end{tikzpicture}

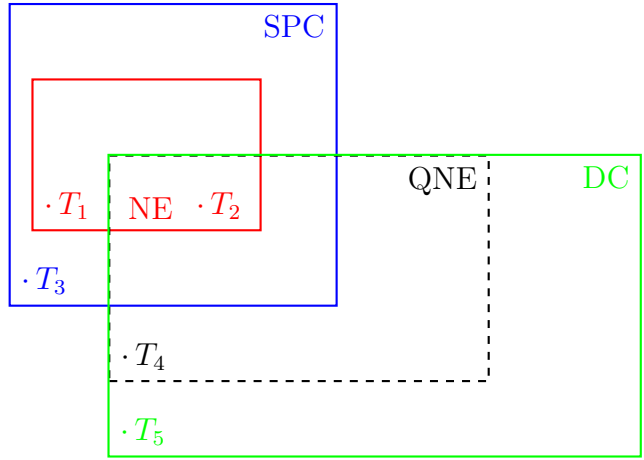
Figure 1. Diagram of the relationships between the classes $\mathcal{NE}$, $\mathcal{QNE}$, $\mathcal{SPC}$ and $\mathcal{DC}$
\medskip

\noindent
The diagram in Figure 1 is taken from Berinde \cite{Ber24}, where the mappings $T_1$-$T_5$ that differentiate the four classes of mappings are also considered in Examples 2.1-2.5 \cite{Ber24}.
%{\bf Figure 1.} Diagram of the relationships between the classes $\mathcal{NE}$, $\mathcal{QNE}$, $\mathcal{SPC}$ and $\mathcal{DC}$

%\section{The technique of enriching contractive type mappings}

\section{Some classes of enriched mappings}\label{s3}

Although chronologically, the first class of enriched mappings introduced in literature was the one corresponding to nonexpansive mappings, we start our presentation with enriched  Banach contractions, which are the mappings appearing in the famous Banach contraction mapping principle - the foundation stone of metrical fixed point theory.

\subsection{Enriched contractions in Banach spaces}
 \indent
 
 The concept of {\em enriched contraction} was introduced and studied in Berinde and P\u acurar \cite{BPac20}.
 \begin{definition}[\cite{BPac20}] \label{def-B}
Let $(X,\|\cdot\|)$ be a linear normed space. A mapping $T:X\rightarrow X$ is said to be an {\it enriched contraction} if there exist $b\in[0,+\infty)$ and $\theta\in[0,b+1)$ such that
\begin{equation} \label{cond-Banach}
\|b(x-y)+Tx-Ty\|\leq \theta \|x-y\|,\forall x,y \in X.
\end{equation}
To indicate the constants involved in \eqref{cond-Banach} we shall also call  $T$ a  $(b,\theta$)-{\it enriched contraction}. 
\end{definition}

\begin{example}[\cite{BPac20}] \label{ex1}
\indent

(1) A Banach contraction $T$  satisfies \eqref{cond-Banach} with $b=0$ and $\theta=c\in [0,1)$.

(2) Let $X=[0,1]$ be endowed with the usual norm and let $T:X\rightarrow X$ be defined by $Tx=1-x$, for all $x\in [0,1]$. Then   $T$ is not a Banach contraction but $T$ is a $(b,1-b)$-enriched contraction for any $b\in(0,1)$. 
\end{example}

An important fixed point theorem and convergence result for enriched contractions is stated in the next theorem.

\begin{theorem} [\cite{BPac20}]  \label{th-B}
Let $(X,\|\cdot\|)$ be a Banach space and $T:X\rightarrow X$ a $(b,\theta$)-{\it enriched contraction}. Then

$(i)$ $Fix\,(T)=\{p\}$;

$(ii)$ There exists $\lambda\in (0,1]$ such that the iterative method
$\{x_n\}^\infty_{n=0}$, given by
\begin{equation} \label{eq3a}
x_{n+1}=(1-\lambda)x_n+\lambda T x_n,\,n\geq 0,
\end{equation}
converges to p, for any $x_0\in X$;

$(iii)$ The following estimate holds
\begin{equation}  \label{3.2-1B}
\|x_{n+i-1}-p\| \leq\frac{c^i}{1-c}\cdot \|x_n-
x_{n-1}\|\,,\quad n=0,1,2,\dots;\,i=1,2,\dots,
\end{equation}
where $c=\dfrac{\theta}{b+1}$.
\end{theorem}

\begin{remark}
In the particular case $b=0$, by Theorem \ref{th-B} we get the classical Banach contraction mapping principle  in the original setting of a Banach space.
\end{remark}

It is possible to establish a Maia type fixed point theorem for enriched contractions defined on a linear vector space, by endowing it with a metric $d$ which is subordinated to a norm $\|\cdot\|$. The next result has been established in Berinde \cite{Ber22}.

\begin{theorem} [\cite{Ber22}] \label{th-Maia}
Let $X$ be a linear vector space endowed with a metric $d$ and a norm $\|\cdot \|$ satisfying the condition
\begin{equation}\label{norm-metric}
d(x,y)\leq \|x-y\|,\,\textnormal{ for all } x,y\in X.
\end{equation}
Suppose

(i) $(X,d)$ is a complete metric space;

(ii) $T:X\rightarrow X$ is continuous with respect to $d$;

(iii)  $T$ is an enriched contraction with respect to $\|\cdot\|$, that is, there exist $b\in[0,+\infty)$ and $\theta\in[0,b+1)$ such that
\begin{equation}  \label{eq1.1a}
\|b(x-y)+Tx-Ty\|\leq \theta \|x-y\|,\forall x,y \in X.
\end{equation}

Then

$(i)$ $Fix\,(T)=\{p\}$, for some $p\in X$;

$(ii)$ There exists $\lambda\in (0,1]$ such that the iterative method
$\{x_n\}^\infty_{n=0}$, given by
\begin{equation}\label{iteratie}
x_{n+1}=(1-\lambda)x_n+\lambda T x_n,\,n\geq 0,
\end{equation}
converges in $(X,d) $ to $p$, for any $x_0\in X$;

$(iii)$ The estimate
\begin{equation}\label{estimare}
d(x_n,p)\leq  \frac{c^n}{1-c}\cdot \|x_{1}-x_{0}\|,\,n\geq 1
\end{equation}
and
\begin{equation}\label{estimare-1}
d(x_n,p)\leq \frac{c}{1-c}\cdot \|x_{n}-x_{n-1}\|,\,n\geq 1,
\end{equation}
hold with $c=\dfrac{\theta}{b+1}$.
\end{theorem}

\subsection{Enriched Kannan mappings  in Banach spaces}
\indent
 
 The concept of {\em enriched Kannan mapping} was introduced and studied in Berinde and P\u acurar \cite{BPac21b}, where some applications for solving split feasibility and variational inequality problems are also presented.
\begin{definition} [\cite{BPac21b}]\label{def2}
Let $(X,\|\cdot\|)$ be a linear normed space. A mapping $T:X\rightarrow X$ is said to be an {\it enriched Kannan mapping} if there exist  $a\in[0,1/2)$ and $k\in[0,\infty)$ such that
\begin{equation} \label{cond_eKannan}
\|k(x-y)+Tx-Ty\|\leq a \left[\|x-Tx\|+\|y-Ty\|\right],\textnormal{ for all } x,y \in X.
\end{equation}
To indicate the constants involved in \eqref{cond_eKannan} we shall also call  $T$ a  $(k,a$)-{\it enriched Kannan mapping}. 
\end{definition}

\begin{example}  \label{ex2a}
\indent

(1) Any Kannan mapping is a $(0,a)$-enriched Kannan mapping, i.e., it satisfies \ref{cond_eKannan} with $k=0$.

(2) Let $X=[0,1]$ be endowed with the usual norm and $T:X\rightarrow X$ be defined by $Tx=1-x$, for all $x\in [0,1]$. 
Then $T$ is not a Kannan mapping but $T$ is an enriched Kannan mapping ($T$ is also nonexpansive). 
\end{example}

The next result provides a convergence theorem for the Krasnoselskij iterative method used to approximate the fixed points of {\it enriched} Kannan mappings.

\begin{theorem} [\cite{BPac21b}] \label{th-K}
Let $(X,\|\cdot\|)$ be a Banach space and $T:X\rightarrow X$ a $(k,a$)-{\it enriched Kannan mapping}. Then

$(i)$ $Fix\,(T)=\{p\}$;

$(ii)$ There exists $\lambda\in (0,1]$ such that the iterative method
$\{x_n\}^\infty_{n=0}$, given by
\begin{equation} \label{eq3a}
x_{n+1}=(1-\lambda)x_n+\lambda T x_n,\,n\geq 0,
\end{equation}
converges to p, for any $x_0\in X$;

$(iii)$ The following estimate holds
\begin{equation}  \label{3.2-1K}
\|x_{n+i-1}-p\| \leq\frac{\delta^i}{1-\delta}\cdot \|x_n-
x_{n-1}\|\,,\quad n=0,1,2,\dots;\,i=1,2,\dots
\end{equation}
where $\delta=\frac{a}{1-a}$.
\end{theorem}

\begin{remark}
The notion of enriched Kannan mapping has been extended to {\em enriched Bianchini mapping} for which corresponding existence and approximation results that generalize Theorem \ref{th-K} were also established in Berinde and P\u acurar \cite{BPac21b}.
\end{remark}

\subsection{Enriched \' Ciri\' c-Reich-Rus contractions in Banach spaces}\indent
\medskip

It is possible to unify and extend Theorems \ref{th-B} and \ref{th-K} from the previous sections and thus obtain a fixed point theorem for the so called {\em enriched \' Ciri\' c-Reich-Rus contractions}. This concept has been first introduced in Berinde and P\u acurar \cite{BPac21a}, in a particular case, and then was improved to the currentt version in Berinde and P\u acurar \cite{BPac22}.

\begin{definition} [\cite{BPac22}] \label{def-CRR}
Let $(X,\|\cdot\|)$ be a linear normed space and $T:X\rightarrow X$  a self mapping. $T$ is  a $(k,a, b$)-{\it enriched \' Ciri\' c-Reich-Rus contraction} if, for some  $k\in[0,\infty)$ and $a, b\geq 0$, satisfying $\dfrac{a}{k+1}+2b<1$, the following condition holds.
\begin{equation} \label{cond_eCRRg}
\|k(x-y)+Tx-Ty\|\leq a \|x-y\|+b \left(\|x-Tx\|+\|y-Ty\|\right),\textnormal{ for all } x,y \in X.
\end{equation}
\end{definition}

\begin{remark}
\indent

1) A \' Ciri\' c-Reich-Rus contraction satisfies \eqref{cond_eCRRg} with $k=0$. 

2)  If $b=0$, then from  \eqref{cond_eCRRg} we obtain the contraction condition \eqref{cond-Banach} that defines enriched contractions, with $k\in[0,+\infty)$ and $a\in[0,k+1)$.  

3) If $a=0$, then from  \eqref{cond_eCRRg} we obtain the contraction condition \eqref{cond_eKannan} satisfied by an enriched Kannan mapping.

\end{remark}

\begin{theorem} [\cite{BPac22}]  \label{th-CRR}
Let $(X,\|\cdot\|)$ be a Banach space and $T:X\rightarrow X$ a $(k,a,b$)-{\it enriched \' Ciri\' c-Reich-Rus contraction} in the sense of Definition \ref{def-CRR}. Then

$(i)$ $Fix\,(T)=\{p\}$;

$(ii)$ There exists $\lambda\in (0,1]$ such that the iterative method
$\{y_n\}^\infty_{n=0}$, given by
\begin{equation}\label{iteratie}
y_{n+1}=(1-\lambda)y_n+\lambda T y_n,\,n\geq 0,
\end{equation}
converges to $p$, for any $y_0\in X$;

$(iii)$ The following estimates hold
\begin{equation}\label{estimare}
\|y_{n}-p\| \leq 
\begin{cases} 
\alpha^n \cdot \|y_0-p\|,n\geq 0\\ \medskip\dfrac{\alpha}{1-\alpha}\cdot \|y_n-
y_{n-1}\|,n\geq 1
\end{cases}
\end{equation}
where $\alpha=\dfrac{a+(k+1)b}{(k+1)(1-b)}$.

\end{theorem}

\begin{remark}
\indent

As mentioned before, a preliminary version of the concept of enriched \' Ciri\' c-Reich-Rus contraction in Definition \ref{def-CRR} has been introduced and studied in Berinde and P\u acurar \cite{BPac21a}.
\end{remark}

\subsection{Enriched Chatterjea mappings in Banach spaces}
\indent
 
 The notion of {\em enriched Chatterjea mapping} was introduced and studied in Berinde and P\u acurar \cite{BPac21}.
\begin{definition}[\cite{BPac21}] \label{thChat}
Let $(X,\|\cdot\|)$ be a linear normed space. A mapping $T:X\rightarrow X$ is said to be an {\it enriched Chatterjea mapping} if there exist $b\in[0,1/2)$ and $k\in[0,+\infty)$  such that
$$
\|k(x-y)+Tx-Ty\|\leq  b \left[\|(k+1)(x-y)+y-Ty\|+\right.
$$
\begin{equation} \label{Def_eChatterjea}
\left.+\|(k+1)(y-x)+x-Tx\|\right]\|,\forall x,y \in X.
\end{equation}
To indicate the constants involved in \eqref{Def_eChatterjea} we shall call  $T$ a $(k,b$)-{\it enriched Chatterjea mapping}. 
\end{definition}

\begin{example} \label{ex1a-C}
\indent

1) A Chatterjea mapping satisfies \eqref{Def_eChatterjea} with $k=0$.

2) All Banach contractions with constant $c<\dfrac{1}{3}$, all Kannan mappings with Kannan constant $a<\dfrac{1}{4}$ and all Chatterjea mappings are enriched Chatterjea mappings, i.e.,  they satisfy \eqref{Def_eChatterjea} with $k=0$. 

3) $T$ in Example \ref{ex2a} (2) is also an enriched Chatterjea mapping. 

\end{example}

\begin{theorem} [\cite{BPac21}]  \label{th-C}
Let $(X,\|\cdot\|)$ be a Banach space and $T:X\rightarrow X$ a $(k,b$)-{\it enriched Chatterjea mapping}. Then

$(i)$ $Fix\,(T)=\{p\}$;

$(ii)$ There exists $\lambda\in (0,1]$ such that the iterative method
$\{x_n\}^\infty_{n=0}$, given by
\begin{equation} \label{3aa}
x_{n+1}=(1-\lambda)x_n+\lambda T x_n,\,n\geq 0,
\end{equation}
converges to p, for any $x_0\in X$;

$(iii)$ The following estimate holds
\begin{equation}  \label{3.2-1C}
\|x_{n+i-1}-p\| \leq\frac{\delta^i}{1-\delta}\cdot \|x_n-
x_{n-1}\|\,,\quad n=0,1,2,\dots;\,i=1,2,\dots
\end{equation}
where $\delta=\dfrac{b}{1-b}$.
\end{theorem}

\subsection{Enriched almost contractions in Banach spaces}
\indent
 
 The notion of {\em enriched almost contraction} has been introduced and studied in Berinde and P\u acurar \cite{BPac23}. It is very general and unifies and extend all the previous concepts of enriched contractive type mappings.

\begin{definition} [\cite{BPac23}] \label{def-AC}
Let $(X,\|\cdot\|)$ be a linear normed space. A mapping $T:X\rightarrow X$ is said to be an {\it enriched almost contraction} if there exist $b\in[0,\infty)$, $\theta\in (0,b+1)$ and $L\geq 0$ such that
\begin{equation} \label{cond-AC}
\|b(x-y)+Tx-Ty\|\leq \theta \|x-y\|+L\|b(x-y)+Tx-y\|,
\end{equation}
for all $x,y \in X$.
To indicate the constants involved in \eqref{cond-AC} we shall also call $T$ as an  {\it enriched} $(b,\theta, L)$-{\it  almost contraction}. 
\end{definition}

\begin{example}[ ] \label{ex-AC}
\indent

1) Any $(\delta,L)$-almost contraction is an enriched $(0,\delta,L)$-almost contraction, i.e., it satisfies \eqref{cond-AC} with $b=0$ and $\theta=\delta$;

2) Any $(b,\theta)$-enriched  contraction  is an enriched $\left(b,\theta, 0\right)$-almost contraction;

3) Any  $(k,a$)-{\it enriched Kannan mapping}  is an enriched $\left(k, \dfrac{a}{1-a}, \dfrac{2 a}{1-a}\right)$-almost contraction;

4) Any $(k,b$)-{\it enriched Chatterjea mapping}  is an enriched $\left(k, \dfrac{b}{1-b}, \dfrac{2 b}{1-b}\right)$-almost contraction.
\end{example}

The next result unifies all main results in the previous subsections, i.e., Theorem \ref{th-B}, Theorem \ref{th-K}, Theorem \ref{th-CRR} and Theorem \ref{th-C}, and provides a Krasnoselskij method for approximating the fixed points of enriched almost contractions.

\begin{theorem} [\cite{BPac23}] \label{th-AC}
Let $(X,\|\cdot\|)$ be a Banach space and let $T:X\rightarrow X$ be a $(b,\theta,L)$-almost contraction.

Then

$1)$ $Fix\,(T)\neq\emptyset$;

$2)$ For any $x_0\in X$, there exists $\lambda\in (0,1)$ such that the Krasnoselskij iteration 
$\{x_n\}^\infty_{n=0}$, defined by,  
\begin{equation} \label{eq3a}
x_{n+1}=(1-\lambda)x_n+\lambda T x_n,\,n\geq 0,
\end{equation}
converges to some $x^\ast\in Fix\,(T)$, for any $x_0\in X$;

$3)$ The  estimate \eqref{3.2-1C} holds
%\begin{equation}  \label{3.2-1}
%\|x_{n+i-1}-x^\ast\| \leq\frac{\delta^i}{1-\delta}\, \|x_n-
%x_{n-1}\|\,,\quad n=0,1,2,\dots;\,i=1,2,\dots,
%\end{equation}
with $\delta=\dfrac{\theta} {b+1}$.
\end{theorem}

The next example is remarkable and illustrates the great generality of enriched almost contractions and therefore of Theorem \ref{th-AC} itself.

\begin{example} [\cite{BPac23}]  \label{ex-AC2}

Let $X=\left[0,\frac{4}{3}\right]$ with the usual norm and $T: X\rightarrow X$ be given by
\begin{equation} \label{eac}
Tx=
\begin{cases}
1-x, \textnormal{ if } x\in \left[0,\frac{2}{3}\right)\\
\smallskip
2-x, \textnormal{ if } x\in \left[\frac{2}{3}, \frac{4}{3}\right].
\end{cases}
\end{equation}
Then $Fix\,(T)=\left\{\dfrac{1}{2},1\right\}$ and:

1) $T$ is a $(1,\theta, 3)$-enriched almost contraction, for any $\theta\in (0,2)$;

2) $T$ is not an almost contraction;

3) $T$ does not belong to the classes of enriched contractions, enriched Kannan mappings or enriched Chatterjea mappings;

4) $T$ is neither nonexpansive nor quasi-nonexpansive;

5)  $T$ is not an enriched nonexpansive mapping.

\end{example}

\begin{remark}

It is important to note that, in view of Theorems \ref{th-B}-\ref{th-C}, enriched contractions,  enriched Kannan mappings, enriched \' Ciri\' c-Reich-Rus contractions and enriched Chatterjea mappings have all a unique fixed point, while enriched almost contractions - which include all these classes of mappings - could have two or more fixed points.

\end{remark}

\subsection{Enriched $\varphi$-contractions  in Banach spaces} %in convex metric spaces
\indent
 
 The notion of {\em enriched $\varphi$-contraction}  has been introduced and studied in Berinde et al. \cite{Ber22a}.
 
 A function $\varphi :\mathbb{R}_{+}\rightarrow \mathbb{R}_{+}$ is said to be a {\it comparison
function} (see for example \cite{Ber07}), if  the following two conditions hold:

(i$_{\varphi }$) \thinspace \thinspace \thinspace $\varphi $ is
nondecreasing, i.e., $t_{1}\leq t_{2}$ implies $\varphi
(t_{1})\leq \varphi (t_{2}).$

(ii$_{\varphi }$)\thinspace \thinspace \thinspace \thinspace $\{\varphi
^{n}(t)\}$ converges to 0 for all $t\geq 0$.

It is obvious that any comparison function also possesses the following property: 

(iii$_{\varphi }$) $\varphi(t)<t$, for $t>0$.
\smallskip

Some examples of comparison functions are the following:
$$
\varphi(t)=\frac{t}{t+1}, t\in [0,\infty);\, \varphi(t)=\frac{t}{2}, t\in [0,1] \textnormal{ and }  \varphi(t)=t-\frac{1}{3}, t\in (1,\infty);\,
$$
(one can note that a comparison function is not necessarily continuous).

\begin{definition} [\cite{Ber22a}]\label{def1-phi}
Consider a linear normed space $(X,\|\cdot\|)$ and let $T:X\rightarrow X$ be a self mapping. $T$ is said to be an {\it enriched $\varphi$-contraction} if one can find a constant  $b\in[0,+\infty)$ and a comparison function $\varphi$ such that
\begin{equation} \label{eq3phi}
\|b(x-y)+Tx-Ty\|\leq (b+1)\varphi (\|x-y\|),\forall x,y \in X.
\end{equation}
We shall also call $T$ a  $(b,\varphi$)-{\it enriched contraction}. 
\end{definition}

\begin{example} \label{ex1-phi}
\indent

1) Any  $(b,\theta$)-{\it enriched contraction} is an enriched $\varphi$-contraction with $\varphi(t)=\dfrac{\theta}{b+1}\cdot  t$. 

2) Any $\varphi$-contraction is a $(0,\varphi$)-enriched contraction.

3) Consider $X$ to be the unit interval $[0,1]$ of $\mathbb{R}$ endowed with the usual norm and the function $T:X\rightarrow X$ given by $Tx=1-x$, for all $x\in [0,1]$. Then  $T$ is neither a contraction nor a $\varphi$-contraction but $T$ is an enriched $\phi$-contraction (as it is an enriched contraction). 
\end{example}

\begin{theorem}[\cite{Ber22a}]\label{th4-phi}
Let $(X,\|\cdot\|)$ be a Banach space and $T:X\rightarrow X$ an enriched $(b,\varphi)$-contraction. Then

$(i)$ $Fix\,(T)=\{p\}$;

$(ii)$ There exists $\lambda\in (0,1]$ such that the iterative method
$\{x_n\}^\infty_{n=0}$, given by
\begin{equation} \label{eq3a}
x_{n+1}=(1-\lambda)x_n+\lambda T x_n,\,n\geq 0,
\end{equation}
and $x_0\in X$ arbitrary, converges strongly to $p$;
\end{theorem}

\begin{remark}
\indent

An enriched $(b,\varphi)$-contraction with $b=0$ is a usual $\varphi$-contraction, a concept that was studied previously in Berinde \cite{Ber97},  \cite{Ber03}, \cite{Ber04} and many other papers.

\end{remark}

Now, let us consider the auxiliary functions $\psi:\mathbb{R}_+\rightarrow [0,1)$ satisfying the following property:
\smallskip

$(g)$ If $\{t_n\}\subset \mathbb{R}_+$ and $\psi(t_n)\rightarrow 1$ as $n\rightarrow \infty$, then $t_n\rightarrow 0$ as $n\rightarrow \infty$.

\smallskip

Let  $\mathcal{P}$ denote the set of all auxiliary functions $\psi$ satisfying condition $(g)$ above. It is easy to check that $\mathcal{P}\neq \emptyset$, as the function $\psi(t)=\exp(-t)$, for $t\geq 0$, belongs to $\mathcal{P}$.

The next result is a very general fixed point theorem, that includes many other fixed point results as particular cases, see Berinde and P\u acurar  \cite{BPac21c}.

\begin{theorem}[\cite{BPac21c}]\label{th5-phi}
Let $(X,\|\cdot\|)$ be a Banach space and let $T:X\rightarrow X$ be an enriched $\psi$-contraction, i.e., a mapping for which there exists a function $\psi\in\mathcal{P}$ such that
\begin{equation} \label{ger}
\|b(x-y)+Tx-Ty\|\leq (b+1)\psi (\|x-y\|) \|x-y\|,\forall x,y \in X.
\end{equation}
Then,  

$(i)$ $Fix\,(T)=\{p\}$, for some $p\in X$.

$(ii)$ The sequence $\{x_n\}^\infty_{n=0}$ obtained from the iterative process
\begin{equation} \label{eq3w}
x_{n+1}=(1-\lambda)x_n+\lambda T x_n,\,n\geq 0,
\end{equation}
and $x_0\in X$ arbitrary, converges strongly to $p$.
\end{theorem}

\subsection{Cyclic enriched $\varphi$-contractions in Banach spaces}
\indent

The class of enriched $\varphi$-contractions has been extended further to {\em cyclic enriched $\varphi$-contractions} in Berinde et al. \cite{Ber22a}. To present it, we need the following prerequisites, see Rus \cite{Rus05}.

Let $X$ be a nonempty set, $m$ a positive integer and $T:X\rightarrow X$ an operator. By definition, $\bigcup_{i=1}^{m} X_i$ is a {\em cyclic representation of $X$ with respect to $T$} if 

$(i)$ $X_i\neq \emptyset$, $i=1,2,...,m$; 

$(ii)$ $T(X_1)\subset X_2$,\dots, $T(X_{m-1})\subset X_{m}$, $T(X_m)\subset X_1$.
\smallskip

Let $(X,d)$ be a metric space, $m$ a positive integer, $A_1,\dots,A_m$ nonempty and closed subsets of $X$ and $Y=\bigcup_{i=1}^{m} A_i$. An operator $T:X\rightarrow X$ is called a {\em cyclic $\varphi$-contraction} if 

$(a)$  $\bigcup_{i=1}^{m} A_i$ is a cyclic representation of $Y$ with respect to $T$;

$(b)$ there exists a comparison function $\varphi:\mathbb{R}_{+} \rightarrow \mathbb{R}_{+}$ such that 
\begin{equation}\label {c3}
d(Tx,Ty)\leq \varphi(d(x,y)), 
\end{equation}
for any $x\in A_i$, $y\in A_{i+1}$, $i=1,2,\dots,m$, where $A_{m+1}=A_1$.

\begin{definition}\label{def7}
Consider a linear normed space $(X,\|\cdot\|)$, $T:X\rightarrow X$ be a self mapping and let $\bigcup_{i=1}^{m} A_i$ be a cyclic representation of $X$ with respect to $T$. If one can find a constant  $b\in[0,+\infty)$ and a comparison function $\varphi$ such that
\begin{equation} \label{phi}
\|b(x-y)+Tx-Ty\|\leq (b+1)\varphi (\|x-y\|),\forall x\in A_i\textnormal{ and } \forall y \in A_{i+1}, 
\end{equation}
for $i=1,2,\dots,m$, where $A_{m+1}=A_1$, then $T$ is said to be a {\it cyclic enriched $\varphi$-contraction}.
\end{definition}

\begin{example}\label{ex7}
\indent

1) Any cyclic $\varphi$-contraction is a cyclic enriched $\varphi$-contraction (with $b=0$);

2) Any enriched contraction is a cyclic enriched $\varphi$-contraction (with $m=1$).
\end{example}

A comparison function $\varphi$ is said to be a (c)-comparison function (see \cite{Ber97}) if there exist $k_0\in \mathbb{N}$, $\delta \in (0,1)$ and a convergent series of nonnegative terms $\sum\limits_{k=1}^{\infty} v_k$ such that
\begin{equation}\label{raport}
\varphi^{k+1} (t)\leq \delta \varphi^{k} (t)+v_k,\,k\geq k_0,\,t\in \mathbb{R}_{+}.
\end{equation}

It is known (see for example Lemma 1.1 \cite{PacR10}) that if $\varphi :\mathbb{R}_{+}\rightarrow \mathbb{R}_{+}$ is a (c)-comparison function, then $s :\mathbb{R}_{+}\rightarrow \mathbb{R}_{+}$ defined by
\begin{equation} \label{suma}
s(t)=\sum_{k=1}^{\infty} \varphi^{k} (t),\,t\in \mathbb{R}_{+}, 
\end{equation}
is increasing and continuous at $0$.

\begin{theorem} [\cite{Ber22a}]\label{th8-Cyclic}
Let $(X,\|\cdot\|)$ be a Banach space, $m$ a positive integer, $A_1,\dots,A_m$ nonempty and closed subsets of $X$, $Y=\bigcup_{i=1}^{m} A_i$ and $T:X\rightarrow X$ a cyclic enriched $\varphi$-contraction with $\varphi$ a (c)-comparison function. Then

$(i)$ $T$ has a unique fixed point $p\in \bigcap_{i=1}^{m} A_i$;

$(ii)$ there exists $\lambda\in (0,1]$ such that the iterative method
$\{x_n\}^\infty_{n=0}$, given by
\begin{equation} \label{eq3a}
x_{n+1}=(1-\lambda)x_n+\lambda T x_n,\,n\geq 0,
\end{equation}
and $x_0\in X$ arbitrary, converges strongly to $p$;

$(iii)$ the following estimates hold
$$
\|x_n-p\|\leq s\left(\varphi^n(\|x_0-x_1\|)\right), n\geq 1;
$$
$$
\|x_n-p\|\leq s\left(\varphi(\|x_n-x_{n+1}\|)\right), n\geq 1;
$$

$(iv)$ for any $x\in Y$:
$$
\|x_n-p\|\leq s(\lambda \|x-Tx\|),
$$
where $s$ is defined by \eqref{suma}.
\end{theorem}

\begin{remark}
If $m=1$, then by Theorem \ref{th8-Cyclic} one obtains Theorem \ref{th5-phi} from the previous section.
\end{remark}

\subsection{Enriched contractions in convex metric spaces} %
\indent
 
 The notion of {\em enriched contraction}  in convex metric spaces has been introduced and studied in Berinde and P\u acurar \cite{BPac21c}. To introduce it we need the following
 
 \begin{definition}  \label{def0-CMS}
Let $(X,d)$ be a metric space. A continuous function $W: X\times X\times [0,1]\rightarrow X$ is said to be a {\it convex structure} on $X$ if, for all $x,y\in X$ and any $\lambda\in [0,1]$,
\begin{equation} \label{eq1}
d(u,W(x,y;\lambda))\leq \lambda d(u,x)+(1-\lambda) d(u,y), \textnormal { for any } u\in X.
\end{equation}
A metric space $(X,d)$ endowed with a convex structure $W$ is called a {\it Takahashi convex metric space} and is usually denoted by $(X,d, W)$.
\end{definition}

\begin{remark}
Any linear normed space and each of its convex subsets are convex metric spaces, with the natural convex structure
\begin{equation}\label{convex}
W(x,y;\lambda)=\lambda x+(1-\lambda) y, x,y\in X; \lambda\in [0,1].
\end{equation}
but the reverse is not valid.
\end{remark}

%: there are various examples of convex metric spaces which cannot be embedded in any Banach space
 
 \begin{definition} [\cite{BPac21c}]\label{def1-CMS}
Let $(X, d, W)$ be a convex metric space. A mapping $T:X\rightarrow X$ is said to be an {\it enriched contraction} if there exist $c\in[0,1)$ and $\lambda\in [0,1)$ such that
\begin{equation} \label{eq3a}
d(W(x,Tx;\lambda), W(y,Ty;\lambda))\leq c d(x,y),\textnormal { for all } x,y \in X.
\end{equation}

To specify the parameters $c$ and $\lambda$ involved in \eqref{eq3a}, we   also call $T$ a $(\lambda,c)$-\linebreak enriched contraction.
\end{definition}

\begin{example}
\indent

Any $(0,c)$-enriched contraction is a usual Banach contraction and therefore any enriched contraction. 

\end{example}

The next result is a significant extension of Theorem \ref{th-B} from the case of a Banach space setting to that of an arbitrary complete convex metric space.

\begin{theorem} [\cite{BPac21c}] \label{th1-CMS}
Let $(X,d, W)$ be a complete convex metric space and let $T:X\rightarrow X$ be a $(\lambda,c)$-enriched contraction. Then,     %%

$(i)$ $Fix\,(T)=\{p\}$, for some $p\in X$.

$(ii)$ The sequence $\{x_n\}^\infty_{n=0}$ obtained from the iterative process
\begin{equation} \label{eq3b-CMS}
x_{n+1}=W(x_n,T x_n;\lambda),\,n\geq 0,
\end{equation}
converges to $p$, for any $x_0\in X$.

$(iii)$ The following estimate holds
\begin{equation} \label{3.2-1-CMS}
d(x_{n+i-1},p) \leq\frac{c^i}{1-c}\cdot d(x_n,
x_{n-1})\,  \quad n=1,2,\dots;\,i=1,2,\dots
\end{equation}
\end{theorem}

\subsection{Enriched Pre\v si\' c contractions}

\indent
 
 The notion of {\em enriched Pre\v si\' c contraction} has been introduced and studied in P\u acurar \cite{Pac23}.
 
 \begin{definition} [\cite{Pac23}] 
	Let $(X,+,\cdot)$ be a linear vector space, $k$ a positive integer and $T:X^k\to X$ an operator. For $\lambda_0, \lambda_1, \dots, \lambda_k\geq0$, with $\underset{i=0}{\overset{k}{\sum}}\lambda_i=1$ and $\lambda_k\neq0$, the operator $T_{\lambda}:X^k\to X$, defined by
	
	\begin{equation}\label{Def_Tlambda}
		T_{\lambda}(x_0,x_1,\dots,x_{k-1})=\lambda_0x_0+\lambda_1x_1+\dots+\lambda_{k-1}x_{k-1}+\lambda_kT(x_0,x_1,\dots,x_{k-1})
	\end{equation}
	will be called the averaged mapping corresponding to T.
\end{definition}

\begin{remark}
	One can easily see that, for $k=1$, the above definition reduces to $T_{\lambda}(x_0)=\lambda_0 x_0 + \lambda_1 T(x_0)$, for $x_0\in X$, where $\lambda_0+\lambda_1=1$, that is, the averaged mapping $T_{\lambda}:X\to X$ extensively used in the previous sections.
\end{remark}

\begin{remark}\label{Rem_T.Tlambda}
	As in the case of the averaged mapping corresponding to an operator defined on $X$, it is not difficult to show that $x^*\in X$ is a fixed point of $T^k:X\to X$ if and only if it is a fixed point of the corresponding $T_{\lambda}:X^k\to X$, for some $\lambda_i\geq0$, $i=0,1,\dots,k$, with $\underset{i=0}{\overset{k}{\sum}}\lambda_i=1$ and $\lambda_k\neq0$.
	
	Indeed, supposing $x^*\in X$ such that $T_{\lambda}(x^*,x^*,\dots,x^*)=x^*$, it follows that
	\[\lambda_0x^*+\lambda_1x^*+\dots+\lambda_{k-1}x^*+\lambda_kT(x^*,x^*,\dots,x^*)=x^*,\]
	so
	\[(1-\lambda_k)x^*+\lambda_kT(x^*,x^*,\dots,x^*)=x^*.\]
	Since $\lambda_k\neq0$, it follows immediately that $T(x^*,x^*,\dots,x^*)=x^*.$
	The inverse is obvious.
\end{remark}

%Using the averaged mapping defined above, we can define now a new general class of Pre\v si\'c type operators:

\begin{definition} [\cite{Pac23}] 
	Let $(X,\left\| \cdot \right\|)$ be a linear normed space and $k$ a positive integer. A mapping $T:X^k\to X$ is said to be an enriched Pre\v si\'c operator if there exist $b_i\geq0,i=0,1,\dots,k-1$ and $\theta_i\geq0,i=0,1,\dots,k-1$ with $\overset{k-1}{\underset{i=0}{\sum}}(\theta_i-b_i)<1$ such that:
	\[
	\left\|\overset{k-1}{\underset{i=0}{\sum}}b_i(x_i-x_{i+1})+T(x_0,x_1,\dots,x_{k-1})-T(x_1,x_2,\dots,x_k)\right\|\leq\overset{k-1}{\underset{i=0}{\sum}}\theta_i\left\|x_i-x_{i+1}\right\|,
	\]
	for all $x_0,x_1,\dots,x_k\in X$.
\end{definition}

\begin{remark}
\indent 

	1) For $k=1$ this reduces to the definition of an enriched Banach contraction;%, see \cite{Berinde.Pacurar_E.Banach}.

	2) If $b_0=b_1=\dots=b_{k-1}=0$ in the above definition, then we obtain the definition of a Pre\v si\' c operator, see P\u acurar \cite{Pac23}.
\end{remark}

The next results states that an enriched Pre\v si\'c operator posseses a unique fixed point, which can be obtained by means of some appropriate iterative methods. 

\begin{theorem} [\cite{Pac23}] \label{Th_EPresic}
	Let  $(X,\left\| \cdot \right\|)$ be a Banach space, $k$ a positive integer and $T:X^k\to X$ an enriched Pre\v si\' c operator with constants  $b_i,\theta_i, i=0,1,\ldots, k-1$. Then:
	\begin{itemize}
		\item[1)] $T$ has a unique fixed point $x^*\in X$ such that $T(x^*,x^*,\dots,x^*)$;
		\item[2)] There exists $a\in(0,1]$ such that the iterative method $\{y_n\}_{n\geq0}$ given by
		\[
		y_n=(1-a)y_{n-1}+aT(y_{n-1},y_{n-1},\dots,y_{n-1}),n\geq1,
		\]
		converges to the unique fixed point $x^*$, starting from any initial point $y_0\in X$.
		
		\item[3)] There exist $\lambda_0, \lambda_1, \dots, \lambda_k\geq0$ with $\underset{i=0}{\overset{k}{\sum}}\lambda_i=1$ and $\lambda_k\neq0$ such that the iterative method $\{x_n\}_{n\geq0}$ given by
		\[
		x_n=\lambda_0x_{n-k}+\lambda_1x_{n-k+1}+\dots+\lambda_{k-1}x_{n-1}+\lambda_kT(x_{n-k},x_{n-k+1},\dots,x_{n-1})
		\]
		or simply
		\[
		x_n=T_{\lambda}(x_{n-k},x_{n-k+1},\dots,x_{n-1}), n\geq1,
		\]
		converges to $x^*$, for any initial points $x_0,x_1,\dots,x_{k-1}\in X$.
		
	\end{itemize}
\end{theorem}

\section{Enriched nonexpansive mappings} \label{s4}

Before proceeding with the exposure of the class of enriched nonexpansive mappings, we recall the following related concepts.

\subsection{Enriched nonexpansive mappings in Hilbert spaces}
\indent
 
 The notion of {\em enriched nonexpansive mapping} has been introduced and studied in Berinde  \cite{Ber19}.
\begin{definition}[\cite{Ber19}]\label{def0}
Let $(X,\|\cdot\|)$ be a linear normed space. A mapping $T:X\rightarrow X$ is said to be an {\it enriched nonexpansive mapping} if there exists $b\in[0, \infty)$  such that
\begin{equation} \label{eq3-NE}
\|b(x-y)+Tx-Ty\|\leq (b+1) \|x-y\|,\forall x,y \in X.
\end{equation}
To indicate the constant involved in \eqref{eq3-NE} we shall also call $T$ as a  $b$-{\it enriched nonexpansive mapping}. 
\end{definition}

It is important to note that inequality \eqref{eq3-NE} in Definition \ref{def0} was derived from \eqref{b2} by denoting $b=\dfrac{1}{\lambda}-1$.

Any nonexpansive mapping $T$ is an enriched nonexpansive mapping, i.e., it satisfies \eqref{eq3-NE} with $b=0$, but   the reverse is not true, as shown by the next example.

\begin{example}[\cite{Ber19}] \label{ex1-NE}
\indent

Let $X=\left[\dfrac{1}{2},2\right]$ be endowed with the usual norm and $T:X\rightarrow X$ be defined by $Tx=\dfrac{1}{x}$, for all $x\in \left[\dfrac{1}{2},2\right]$. Then 

(i) $T$ is Lipschitz continuous with Lipschitz constant $L=4$ (and so $T$ is not nonexpansive);

(ii) $T$ is  a $3/2$-enriched nonexpansive mapping.

%(iii)  $Fix\,(T)=\{1\}$.

\end{example}

For the sake of completeness, we recall that a mapping $T:C\rightarrow H$, where $C$ is a bounded closed convex
subset of a Hilbert space $H$, is called \textit{demicompact} 
if it has the property that whenever $\{u_{n}\}$ is a bounded sequence in $H$
and $\{Tu_{n}-u_{n}\}$ is strongly convergent, then there exists a
subsequence $\{u_{n_{k}}\}$ of $\{u_{n}\}$ which is strongly convergent.

The next result states that any enriched nonexpansive mapping which is also demicompact has a nonempty convex fixed point set and that one can approximate its fixed points by means of a Krasnoselskij type iterative scheme.

\begin{theorem} [\cite{Ber19}] \label{th1}
Let $C$ be a bounded closed convex
subset of a Hilbert space $H$ and  $T:C\rightarrow C$ be a $b$-enriched nonexpansive and demicompact mapping. Then the set $Fix\,(T)$ of fixed points of $T$ is a nonempty convex set and there exists $\lambda\in \left(0,1\right)$ such that, for any given $x_0\in C$,  the Krasnoselskij iteration $\{x_n\}_{n=0}^{\infty}$ given by
\begin{equation} \label{eq4}
x_{n+1}=(1-\lambda) x_n+\lambda Tx_n,\,n\geq 0,
\end{equation}
converges strongly to a fixed point of $T$.
\end{theorem}

If we denote by $\mathcal{ENE}$ the class of enriched nonexpansive mappings, then by the previous example it follows that we have the following strict inclusion relationship
$$
\mathcal{NE} \subsetneq \mathcal{ENE}
$$ 
and, by the diagram in Figure 1, we also deduce that 
$$
\mathcal{NE} \subsetneq \mathcal{SPC}.
$$ 
It was then natural to raise the following 
\medskip

{\bf Problem.} Find the relationship between the classes $\mathcal{ENE}$ and $\mathcal{SPC}$.
\medskip

For the case of enriched nonexpansive mappings defined on a real Hilbert space, the answer is given by the next theorem which is a reformulated version of Theorem 8 in Berinde and P\u acurar \cite{BPac21d}.

\begin{theorem} \label{p5}
In a real Hilbert space, $\mathcal{ENE}=\mathcal{SPC}$.
\end{theorem}

\begin{proof}
Let $T\in \mathcal{SPC}$. Then  $T$ satisfies \eqref{strict} with some $k\in (0,1)$.
We have
\begin{equation}\label{b3}
\|x-y-(Tx-Ty)\|^2=\langle x-y-(Tx-Ty), x-y-(Tx-Ty)\rangle
\end{equation}
$$
=\|x-y\|^2-2\langle x-y, Tx-Ty\rangle+\|Tx-Ty\|^2
$$
and so \eqref{strict} is equivalent to
$$
\|Tx-Ty\|^2\leq \frac{1+k}{1-k}\cdot \|x-y\|^2-\frac{2k}{1-k}\cdot \langle x-y,Tx-Ty\rangle.
$$
By adding to both sides of the previous inequality the quantity
$$
\left(\frac{k}{1-k}\right)^2\cdot \|x-y\|^2+ \dfrac{2k}{1-k}\langle x-y,Tx-Ty\rangle,
$$
we deduce that \eqref{strict} is equivalent to
\begin{equation}\label{destept}
\left\|\frac{k}{1-k}(x-y)+Tx-Ty\right\|^2\leq \left(\left(\frac{k}{1-k}\right)^2+\frac{1+k}{1-k}\right) \|x-y\|^2.%+\left(2b-\frac{2k}{1-k}\right)\cdot \langle x-y,Tx-Ty\rangle.
\end{equation}
Now, by denoting $b=\dfrac{k}{1-k}>0$, %it follows that $b^2+\dfrac{1+k}{1-k}=(b+1)^2$ and therefore 
it follows that inequality \eqref{destept} 
is equivalent to 
$$
\|b(x-y)+Tx-Ty\|\leq (b+1) \|x-y\|,\,\forall x,y \in C, 
$$
and this shows that $T\in \mathcal{NE}$. 

The converse follows by runing backwardly the previous implications.
\end{proof}

\subsection{Enriched nonexpansive mappings in Banach spaces}

\begin{remark}
The equality in Theorem \ref{p5} is no more valid if we work in a Banach space. 
The main reason is that in a Banach space we cannot derive the fundamental identity \eqref{b3} in the proof of Theorem \ref{p5}, as it is expressed by means of the inner product in the Hilbert space $H$. 
\end{remark}

Hence in a Banach space the class of enriched nonexpansive mappings and that of of strictly pseudocontractive mappings are independent and therefore the next result is an important generalization of several results in literature established for nonexpansive mappings, e.g., in Browder and Petryshyn \cite{BroP67}. 

\begin{theorem} [\cite{Ber20}] \label{th1a}
Let $C$ be a nonempty bounded closed convex
subset of a uniformly convex Banach space $X$ and let $T:C\rightarrow C$ be a $b$-enriched nonexpansive mapping. Suppose $T$ satisfies Condition I, i.e., 
there exists a nondecreasing function $f:[0,\infty)\rightarrow [0,\infty)$ with the properties $f(0)=0$ and $f(r)>r$, for $r>0$, such that
\begin{equation} \label{eq2}
\|x-Tx\|\geq f(d(x, Fix\,(T)), \forall x\in C,
\end{equation}
where
$$
d(x, Fix\,(T))=\inf \{\|x-z\|: z\in Fix\,(T)\}
$$
is the distance between the point $x$ and the set $Fix\,(T)$. 

Then $Fix\,(T)\neq \emptyset$ and, for any $\lambda\in \left(0,\frac{1}{b+1}\right)$ and for any given $x_0\in C$, the Krasnoselskij iteration $\{x_n\}_{n=0}^{\infty}$ given by
\begin{equation} \label{eq4}
x_{n+1}=(1-\lambda) x_n+\lambda Tx_n,\,n\geq 0,
\end{equation}
converges strongly to a fixed point of $T$.
\end{theorem}

\section{Unsaturated and saturated classes of contractive mappings}\label{s5}

After a close examination of the fixed point results in the Sections \ref{s3} and \ref{s4}, it was noted that the technique of enriching a contractive type mapping $T$, by means of the averaged operator $T_\lambda$, cannot effectively enlarge all classes of contractive mappings. 

This observation suggested us a new interesting concept,  that of  {\it saturated} class of contractive mappings with respect to the averaged operator $T_\lambda$, a notion that has been introduced and studied in Berinde and P\u acurar \cite{BPac21d}.

\begin{definition}[\cite{BPac21d}] \label{bogat}
Let $(X,\|\cdot\|)$ be a linear normed space and let $\mathcal{C}$ be a subset of the family of all self mappings of $X$. A mapping $T:X\rightarrow X$ is said to be  $\mathcal{C}${\em -enriched} or {\em enriched with respect to} $\mathcal{C}$ if there exists $\lambda\in (0,1]$ such  $T_\lambda\in \mathcal{C}$. 

We denote by $\mathcal{C}^{e}$ the set of all enriched mappings with respect to $\mathcal{C}$.
\end{definition}

\begin{remark}
From Definition \ref{bogat} it immediately follows that $\mathcal{C}\subseteq \mathcal{C}^{e}$.
\end{remark}

\begin{definition} [\cite{BPac21d}]\label{sat}
Let $X$ be a linear vector space and let $\mathcal{C}$ be a subset of the family of all self mappings of $X$. If $\mathcal{C}= \mathcal{C}^{e}$, we say that $\mathcal{C}$  is a {\em saturated class of mappings}, otherwise $\mathcal{C}$ is said to be {\em unsaturated}. 
\end{definition}

If we summarize and the results surveyed in the previous two sections of this paper, then we can see that the following classes of mappings:
\begin{itemize}
\item Banach contractions
\item Kannan mappings 
\item \' Ciri\' c-Reich-Rus contractions
\item Chatterjea mappings
\item almost contractions
\item $\varphi$-contractions
\item cyclic enriched $\varphi$-contractions
\item Pre\v si\' c contractions
\end{itemize}
are {\em unsaturated} in the setting of a  Banach space.

Also, from the results surveyed in Section \ref{s4} we infer that
\begin{itemize}
\item  the class of nonexpansive mappings is unsaturated in Hilbert spaces;
\item the class of nonexpansive mappings is unsaturated in Banach spaces
\end{itemize}

Let us denote by $\mathcal{E}$ the enriching operator by the average perturbation of a certain class of self mappings of $X$, i.e., if $\mathcal{C}$ is a class of mappings, then
$$
\mathcal{E}(\mathcal{C})=\mathcal{C}^{e}.
$$ 
It is easy to prove that $\mathcal{E}$ is idempotent, that is, $\mathcal{E}\circ \mathcal{E}=\mathcal{E}$, which means that any class of enriched mappings is {\em saturated}. 

This implies that
\begin{itemize}
\item the class of strictly pseudocontractive mappings is saturated in Hilbert spaces;
\item the class of enriched nonexpansive mappings  is saturated in Hilbert spaces;
\item the class of demicontractive mappings is  is saturated in Hilbert spaces.
\end{itemize}
  
  The last claim follows by Theorem  9 in Berinde and P\u acurar \cite{BPac21d}.

\section{Enriched contractions in quasi-Banach spaces} \label{s6}
\indent

The notion of enriched contraction in the setting of a quasi-Banach space has been introduced and studied in Berinde  \cite{Ber24}. We recall the following prerequisites.

\begin{definition} \label{def2-qB}
A quasi-norm on a real vector space $X$ is a  map $\|\cdot \|: X\rightarrow [0,\infty)$  satisfying the following conditions:
\smallskip

$(QN_0)$ $\|x\|=0$ if and only if $x=0$;

$(QN_1)$ $\|\lambda x\|=|\lambda|\cdot \|x\|$, for all $x\in X$ and $\lambda\in \mathbb{R}$.

$(QN_2)$ $\|x+y\|\leq C\left[\|x\|+\|y\|\right]$, for all $x,y\in X$, where $C\geq 1$ does not depend on $x,y$;
\end{definition}

The pair $(X,\|\cdot \|)$, where $\|\cdot \|$ is a quasi-norm on a real vector space $X$, is said to be  a quasi-normed space. If $(X,\|\cdot \|)$ is complete (with respect to the quasi norm), then  is called a {\it quasi-Banach space}.% and is usually denoted by $(X,d, W)$.

\begin{definition}\label{def3-qB}
Let $(X,\|\cdot\|)$ be a linear quasi-normed space. A mapping $T:X\rightarrow X$ is said to be an {\it enriched contraction} if there exist $b\in[0,+\infty)$ and $\theta\in[0,b+1)$ such that
\begin{equation} \label{eq3-qB}
\|b(x-y)+Tx-Ty\|\leq \theta \|x-y\|,\forall x,y \in X.
\end{equation}
To indicate the constants involved in \eqref{eq3-qB} we shall also call  $T$ a  $(b,\theta$)-{\it enriched contraction}. 
\end{definition}

\begin{remark}
\indent

1) As any Banach space is a quasi-Banach space (with $C=1$), the enriched contractions $T$ in a Banach space introduced in Berinde and P\u acurar \cite{BPac20} are enriched contractions in the sense of Definition \ref{def3-qB}.  

2) It is worth mentioning that, like in the case of Banach spaces, any $(b,\theta$)-enriched contraction is continuous.

\end{remark}

The following result is an extension of Theorem \ref{th-B} from Banach spaces to quasi-Banach spaces.

\begin{theorem} [Berinde  \cite{Ber24}] \label{th1-qB}
Let $(X,\|\cdot\|)$ be a quasi-Banach space and $T:X\rightarrow X$ a $(b,\theta$)-{\it enriched contraction}. Then

$(i)$ $Fix\,(T)=\{p\}$;

$(ii)$ There exists $\lambda\in (0,1]$ such that the iterative method
$\{x_n\}^\infty_{n=0}$, given by
\begin{equation} \label{eq3a-qB}
x_{n+1}=(1-\lambda)x_n+\lambda T x_n,\,n\geq 0,
\end{equation}
converges to p, for any $x_0\in X$;

%$(iii)$ The following estimate holds
%\begin{equation}  \label{3.2-1}
%\|x_{n+i-1}-p\| \leq\frac{c^i}{1-c}\cdot \|x_n-
%x_{n-1}\|\,,\quad n=0,1,2,\dots;\,i=1,2,\dots,
%\end{equation}
%where $c=\dfrac{\theta}{b+1}$.
\end{theorem}

\section{Other developments in the area of enriched contractive type mappings } \label{s7}

The technique of enriching contractive type mappings, applied by the current authors to several classes of mappings, attracted the interest of many other researchers. In the following we summarise to date some of these contributions.

\begin{enumerate}
\item Abbas et al. \cite{Abb21} introduced the notion of {\it enriched quasi-contraction}, as a generalization of the clasical \' Ciri\' c quasi-contraction, and also the class of  {\it enriched weak contraction mappings} and established  and studied the existence and iterative approximation of their fixed points. 

\item Abbas et al. \cite{Abb21a} introduced and studied the class of {\it enriched multivalued contraction mappings} and also considered the data dependence problem and Ulam-Hyers stability of the fixed point problems for enriched multivalued contraction mappings. They also give applications of the obtained results to the problem of the existence of a solution of differential inclusions and dynamic programming.

\item Abbas et al. \cite{Abb22} introduced the concept of {\it generalized enriched cyclic contraction mapping} and obtained existence of fixed points and established convergence results for Krasnoselskij iteration used for approximating fixed points of such mappings. As an application of their results, they also established the existence and uniqueness of an attractor for an iterated function system composed of generalized enriched cyclic contraction mappings.

\item Abbas et al. \cite{Abb23a} introduced the concept of {\it enriched contractive mappings of Suzuki type}. Such a mapping has a fixed point and characterizes the completeness of the underlying normed space.

\item Abbas et al. \cite{Abb22a} introduced the class of {\it enriched interpolative Kannan type operators} on Banach spaces. This class contains the classes of enriched Kannan operators, interpolative Kannan type contraction operators and some other classes of nonlinear operators. They prove a convergence theorem for the Krasnoselskii iteration method to approximate fixed point of the enriched interpolative Kannan type operatorsand, as an application of the main result, solved a variational inequality problem. The same authors propose in \cite{Abb23b} a new class of multi-valued enriched interpolative \' Ciri\' c-Reich-Rus type contraction operators, prove a fixed point result, study the data dependence and Ulam-Hyers stability for these operators and obtain a homotopy result as an application of their results.

\item Hacio\v glu and G\" ursoy \cite{Hac21} introduced multivalued G\' ornicki mappings and various other new types of multivalued enriched contractive mappings, like multivalued enriched Kannan mappings, multivalued enriched Chatterjea mappings, and multivalued enriched \' Ciri\' c-Reich-Rus mappings, and established existence results for the fixed points of these multivalued contractive type mappings by using the fixed point property of the average operator of the mappings.

\item Babu and Mounika \cite{Babu23} defined the classes of {\it enriched Jaggi contraction maps}, {\it enriched Dass and
Gupta contraction maps} and {\it almost $(k, a, b, \lambda)$-enriched CRR contraction maps} in Banach spaces and proved the existence and uniqueness of fixed points of these maps.

\item By using some semi-implicit relations, Mondal et al. \cite {Mon21} introduced  {\it enriched $\mathcal{A}$-contractions} and {\it enriched $\mathcal{A}'$-contractions} and studied the existence of fixed points, the well-posedness and limit shadowing property of the fixed point problem involving these contractions. The same authors obtained laterMaia type results for enriched contractions via implicit relations  \cite{Bera23} , thus extending the results of Berinde \cite{Ber22}.

\item Chandok \cite{Chan22} obtained convergence and existence results of best proximity points for cyclic enriched contraction maps in Takahashi convex metric spaces. 

\item Popescu \cite{Pop21} introduced a new class of Picard operators, called {\it G\' ornicki mappings}, which includes the class of enriched contractions \cite{BPac20}, enriched Kannan mappings \cite{BPac21b}, and enriched Chatterjea mappings \cite{BPac21}, and proved some fixed point theorems for these mappings. However, while the fixed points of enriched contractions can be approximated by means of Krasnoselskij iteration, there is no any approximation result in \cite{Pop21} for the case of G\' ornicki mappings. This rise the challenging problem of finding iterative schemes to approximate the unique fixed  point of a G\' ornicki mapping.

\item Debnath \cite{Deb22} introduced the notion of G\' ornicki-type pair of mappings, establish a criterion for existence and uniqueness of common fixed point for such a pair without assuming continuity of the underlying mappings and also establish a common fixed point result for a pair of enriched contractions.

\item Deshmukh et al. \cite{Desh23}, amongst many other related results for enriched non-expansive maps and enriched generalized non-expansive maps, also give stability results for two iterative  procedures in the class of enriched contractions.

\item Faraji and Radenovi\' c \cite{Far23} established fixed point results for enriched contractions and enriched Kannan contractions in partially ordered Banach spaces.

\item Based on the so-called degree of nondensifiability, Garc\' ia \cite{Gar22} introduced a generalization of the $(b,\theta)$-enriched contractions \cite{BPac20} and established a fixed point existence result for this new class of mappings. From their main result, and under some suitable conditions, they derived a result on the existence of fixed points for the sum of two mappings, one of them being compact.

\item Khan et al. \cite{Khan23} initiate the study of enriched mappings in modular function spaces, by introducing the concepts of {\it enriched $\rho$-contractions} and {\it enriched $\rho$-Kannan mappings} and establishing some results on the existence of fixed points of such mappings in this setting. 

\item By introducing the concept of convex structure in rectangular $G_b$-metric spaces, Li and Cui \cite{Li22}  studied the existence of  fixed points of enriched type contractions in such a space.

\item As a generalization of the main result in  \cite{BPac20},  Marchi\c s \cite{Mar21} obtained some common fixed point theorems under an enriched type contraction condition for two single-valued mappings satisfying a weak commutativity condition in Banach spaces and has shown that  the unique common fixed point of these mappings can be approximated using the Krasnoselskij iteration.

\item By using the idea of the orbital contraction condition given in \cite{Pet21} and considering the second iterate of the mapping in the enriched contraction condition, Nithiarayaphaks and Sintunavarat \cite{Nith23} introduced the class of {\it weak enriched contraction mappings} and approximated their fixed point by Kirk's iterative scheme.

\item  Panicker and Shukla \cite{Pani22}  obtained stability results of fixed point sets for a sequence of enriched contraction mappings in the setting of convex metric spaces, by considering two types of convergence of sequences of mappings, namely, $(\mathcal{G})$-convergence and $(\mathcal{H})$-convergence.

\item Panja et al. \cite{Panj22}  introduced a new non-linear semigroup of enriched Kannan type contractions and proved the existence of a common fixed point on a closed, convex, bounded subset of a real Banach space having uniform normal structure.

\item Among many other related results, Prithvi and Katiyar \cite{Prit23} studied fractals through generalized cyclic enriched \' Ciri\' c-Reich-Rus iterated function systems.

\item Rawat et al. \cite{Raw22} considered enriched ordered contractions in convex noncommutative Banach spaces, while Rawat, Bartwal and Dimri \cite{Raw23}  defined and studied interpolative enriched contractions of Kannan type, Hardy-Rogers type and Matkowski type in the setting of a convex metric space.

\item Ali and  Jubair \cite{Ali23} introduced and studied the so called enriched Berinde nonexpansive mappings, which are related to enriched almost contractions.

\item Anjali and Batra \cite{Anjali} introduced enriched \' Ciri\' c's type and enriched Hardy-Rogers contractions for which they established fixed point theorems in Banach spaces and convex metric spaces. They showed that \' Ciri\' c's type and Hardy-Rogers contractions are unsaturated classes of mappings and also considered Reich and Bianchini contractions, which were shown to be unsaturated classes of mappings, too.

\item Babu and Mounika \cite{Babu23} introduced enriched Jaggi contraction maps, enriched Dass and Gupta contraction maps and almost $(k, a, b, \lambda)$-enriched CRR contraction maps in Banach spaces and established results on the existence and uniqueness of fixed points of these maps. 

\item Zhou et al. \cite{Zhou}  introduced and studied  weak enriched $\mathcal{F}$-contractions, weak enriched $\mathcal{F}'$-contraction, and $k$-fold averaged mapping based on Kirk's iterative algorithm of order $k$ and proved the existence of a unique fixed point of the $k$-fold averaged mapping associated with weak enriched contractions considered.

\end{enumerate}

\section{Conclusions}

In the first part of this paper we presented a few relevant facts about the way in which the technique of enriching contractive mappings was (re-)discovered. 

In the main part of the paper we have exposed the main contributions in the area of enriched mappings established by the authors and their collaborators  by using this technique.

In the last part, we also surveyed some related developments which were authored by other researchers, by considering a list of references to date.

\section*{Acknowledgements}

The first author is grateful to Department of Mathematics and Computer
Science, North University Centre at Baia Mare,  Technical University of Cluj-Napoca (former North University of Baia Mare), for offering him an excellent research environment in the last 35 years.

\vskip 0.5 cm {\it $^{1}$ Department of Mathematics and Computer Science

North University Center at Baia Mare

Technical University of Cluj-Napoca 

Victoriei 76, 430122 Baia Mare ROMANIA

E-mail: vberinde@cunbm.utcluj.ro}

\vskip 0.5 cm {\it $^{2}$ Academy of Romanian Scientists  (www.aosr.ro)

E-mail: vasile.berinde@gmail.com}

\vskip 0.5 cm {\it $^{3}$ Department of Statistics, Analysis, Forecast and Mathematics

Faculty of Economics and Bussiness Administration 

Babe\c s-Bolyai University of Cluj-Napoca,  Cluj-Napoca ROMANIA

E-mail: madalina.pacurar@econ.ubbcluj.ro}

\end{document}